\documentclass[12pt,a4paper]{article}
\usepackage[dvips]{graphicx,color}   
\usepackage[T1]{fontenc}  
\usepackage{times}     
\usepackage{enumerate,amsmath,amsthm,mathrsfs,citesort}
\usepackage{amsfonts}
\usepackage{amssymb}
\usepackage[a4paper,left=2.5cm,right=2cm,top=2.5cm,bottom=2.5cm]{geometry}
\usepackage{picins}
\usepackage{amstext} 
\usepackage{array}   
\usepackage[disable]{todonotes} 
\newcounter{count}
\newcommand{\ma}{\ensuremath{\mu(G)}}

\newtheoremstyle{bthm}{\baselineskip}{\baselineskip}{\slshape}{}{\bfseries}{}{ }{}
\newtheoremstyle{bex}{\baselineskip}{\baselineskip}{}{}{\sffamily}{:}{\newline }{}
\theoremstyle{bthm}
\newtheorem{theorem}{Theorem}[section]
\newtheorem{corollary}[theorem]{Corollary}
\newtheorem{lemma}[theorem]{Lemma}

\newtheorem{problem}[theorem]{Problem}

\newtheorem{conjecture}[theorem]{Conjecture}
\theoremstyle{bex}

\begin{document}
\begin{titlepage}
\title{Bounds for the Chromatic Number of some $pK_2$-Free Graphs}
\author{Athmakoori Prashant$^{1}$, S. Francis Raj$^{2}$ and M. Gokulnath$^{3}$}
\date{{\footnotesize Department of Mathematics, Pondicherry University, Puducherry-605014, India.}\\
{\footnotesize$^{1}$: 11994prashant@gmail.com\, $^{2}$: francisraj\_s@pondiuni.ac.in\, $^3$: gokulnath.math@gmail.com\ }}
\maketitle
\renewcommand{\baselinestretch}{1.3}\normalsize
\begin{abstract}
The concept of $\chi$-binding functions for classes of free graphs has been extensively studied in the past. In this paper, 
we improve the existing $\chi$-binding function for $\{2K_2, K_1 + C_4\}$-free graphs. Also, we find a linear $\chi$-binding function for $\{2K_2, K_2+P_4\}$-free graphs. In addition, we give alternative proofs for the $\chi$-binding function of $\{2K_2,gem\}$-free graphs, $\{2K_2,HVN\}$-free graphs and $\{2K_2, K_5-e\}$-free graphs. Finally,
for $p\geq3$, we find polynomial $\chi$-binding functions for $\{pK_2, H\}$-free graphs where $H\in \{gem, diamond, K_2+P_4, HVN, K_5-e, butterfly, gem^+, dart, K_1 + C_4, C_4, \overline{P_5}\}$.
\end{abstract}
\noindent
\textbf{Key Words:} Chromatic number, $\chi$-binding function, $2K_2$-free graphs and $pK_2$-free graphs. \\
\textbf{2000 AMS Subject Classification:} 05C15, 05C75
\section{Introduction}\label{intro}
All graphs considered in this paper are simple, finite and undirected.
Let $G$ be a graph with vertex set $V(G)$ and edge set $E(G)$.
For any positive integer $k$, a proper $k$-coloring of a graph $G$ is a mapping $c$ : $V(G)\rightarrow\{1,2,\ldots,k\}$ such that for any two adjacent vertices $u,v\in V(G)$, $c(u)\neq c(v)$.
If a graph $G$ admits a proper $k$-coloring, then $G$ is said to be $k$-colorable.
The chromatic number, $\chi(G)$, of a graph $G$ is the smallest $k$ such that $G$ is $k$-colorable.
All colorings considered in this paper are proper.
In this paper, $P_n,C_n$ and $ K_n$ respectively denotes the path, the cycle and the complete graph on $n$ vertices.
For $S,T\subseteq V(G)$, let $N_T(S)=N(S)\cap T$ (where $N(S)$ denotes the set of all neighbors of $S$ in $G$), let $\langle S\rangle$ denote the subgraph induced by $S$ in $G$ and let $[S,T]$ denote the set of all edges with one end in $S$ and the other end in $T$.
If every vertex in $S$ is adjacent with every vertex in $T$, then $ [S, T ]$ is said to be complete.
For any graph $G$, let $\overline{G}$ denote the complement of $G$. 
Let $H\sqsubseteq G$ mean that $H$ is an induced subgraph of $G$.

Let $\mathcal{F}$ be a family of graphs.
We say that $G$ is $\mathcal{F}$-free if it contains no induced subgraph which is isomorphic to a graph in $\mathcal{F}$.
For a fixed graph $H$, let us denote the family of $H$-free graphs by $\mathcal{G}(H)$.
For two vertex-disjoint graphs $G_1$ and $G_2$, the join of $G_1$ and $G_2$, denoted by $G_1+G_2$, is the graph whose vertex set $V(G_1+G_2) = V(G_1)\cup V(G_2)$ and the edge set $E(G_1+G_2) = E(G_1)\cup E(G_2)\cup\{xy: x\in V(G_1),\ y\in V(G_2)\}$.

A clique (independent set) in a graph $G$ is a set of pairwise adjacent (non-adjacent) vertices.
The size of a largest clique (independent set) in $G$ is called the clique number (independence number) of $G$, and is denoted by $\omega(G)\big(\alpha(G)\big)$.
When there is no ambiguity, $\omega(G)$ will be denoted by $\omega$.

A graph $G$ is said to be perfect if $\chi(H)=\omega(H)$, for every induced subgraph $H$ of $G$.
A class $\mathcal{G}$ of graphs is said to be $\chi$-bounded \cite{gyarfas1987problems} if there is a function $f$ (called a $\chi$-binding function) such that $\chi(G)\leq f(\omega(G))$, for every $G\in \mathcal{G}$.
We say that the $\chi$-binding function $f$ is special linear if $f(x)= x+c$, where $c$ is a constant.
There has been extensive research done on $\chi$-binding functions for various graph classes.
See for instance, \cite{schiermeyer2019polynomial,karthick2016vizing,randerath2004vertex}.
The study on $\chi$-binding functions and $\chi$-bounded graphs was initiated by A. Gy{\'a}rf{\'a}s in \cite{gyarfas1987problems}.
Before we move further on $\chi$-binding functions, let us state a famous result  by P. Erd\H{o}s. 
\begin{theorem}[\cite{erdos1959graph}]\label{erdos} 
For any positive integers $k,l\geq 3$, there exists a graph $G$ with girth at least $l$, where girth of $G$ is the length of the shortest cycle in $G$ and $\chi(G)\geq k$.
\end{theorem}
By Theorem \ref{erdos}, A. Gy{\'a}rf{\'a}s conjectured that the only possibility for $\mathcal{G}(H)$ to  have a $\chi$-binding function is when $H$ is acyclic. 
\begin{conjecture}[\cite{gyarfas1987problems}]
$\mathcal{G}(H)$ is $\chi$-bound for every fixed forest $H$.
\end{conjecture}
The following open problem was also posed by A. Gy{\'a}rf{\'a}s in \cite{gyarfas1987problems}.
\begin{problem}[\cite{gyarfas1987problems}]\label{prob1}
What is the {order of magnitude} of the smallest $\chi$-binding function for $\mathcal{G}(2K_2)$?
\end{problem}
Motivated by Problem \ref{prob1}, we started looking at $\mathcal{G}(2K_2)$.
In \cite{wagon1980bound}, S. Wagon showed that for any $p\in\mathbb{N}$, the class of $pK_2$-free graphs admit $O(\omega^{2p-2})$ $\chi$-binding function. 
In particular, he showed that the $\chi$-binding function for $2K_2$-free graphs is $\binom{\omega+1}{2}$.
The best known lower bound for $2K_2$-free graphs given by A. Gy{\'a}rf{\'a}s \cite{gyarfas1987problems} is $\frac{R(C_4,K_{\omega(G)+1})}{3}$, where $R(C_4,K_t)$ denotes a Ramsey number, that is, $R(C_4,K_t)$ is the smallest integer $k$ such that every graph of order at least $k$ either contains a clique of size $t$ or a $2K_2$. 
F. Chung in \cite{chung1980coverings} proved that $R(C_4,K_t)$ is at least $t^{1+\epsilon}$ for some $\epsilon>0$ and hence non-linear. 
In addition, interestingly C. Brause et al. in \cite{brause2019chromatic} proved that the class of $\{2K_2,H\}$-free graphs, where $H$ is a graph with $\alpha(H)\geq 3$, does not admit a linear $\chi$-binding function.
This raised a natural question ``whether there exists any subfamilies of $2K_2$-free graphs which has a linear $\chi$-binding function?''.
In an attempt to answer this question, T. Karthick and S. Mishra in \cite{karthick2018chromatic} proved that the families of $\{2K_2, H\}$-free graphs, where $H\in \{HVN, diamond, gem, K_1 + C_4, \overline{P_5}, \overline{P_2 \cup P_3}, K_5-e\}$ admit a special linear $\chi$-binding functions.
In particular, they proved that $\{2K_2, gem\}$-free, $\{2K_2, HVN\}$-free, $\{2K_2, K_5-e\}$-free and $\{2K_2, K_1+C_4\}$-free admit the $\chi$-binding functions $\omega(G)+1$, $\omega(G)+3$, $\omega(G)+4$ and $\omega(G)+5$ respectively.
In  \cite{brause2019chromatic}, C. Brause et al. improved the $\chi$-binding function  for $\{2K_2,gem\}$-free graphs to $\max\{3, \omega(G)\}$.
They also proved that for $s\neq 1$ or $\omega(G)\neq 2$, the class of $\{2K_2, (K_1\cup K_2)+K_s\}$-free with $\omega(G)\geq 2s$ is perfect and for  $r\geq 1$, the class of $\{2K_2, 2K_1+K_r\}$-free graphs with $\omega(G)\geq 2r$  is perfect.
Clearly when $s=2$ and $r=3$, $(K_1\cup K_2)+K_s\cong HVN$ and $2K_1+K_r\cong K_5-e$ which implies that the class of $\{2K_2,HVN\}$-free and $\{2K_2,K_5-e\}$-free graphs are perfect for $\omega(G)\geq 4$ and $\omega(G)\geq 6$ respectively which improved the bounds given in \cite{karthick2018chromatic}.

In this paper, we are interested in seeing how these bounds would change when we consider $\{pK_2,H\}$-free graphs, $p\geq2$ and $H\in \{gem, diamond, K_2+P_4, HVN, K_5-e, butterfly, gem^+,$ $dart, K_1 + C_4, C_4, \overline{P_5}\}$. 

Throughout this paper, we use a particular partition of the vertex set of a graph $G$ as defined initially by S. Wagon in \cite{wagon1980bound} and later improved by A. P. Bharathi et al. in \cite{bharathi2018colouring} as follows.
Let $A=\{v_1,v_2,\ldots,v_{\omega}\}$ be a maximum clique of $G$.
Let us define the lexicographic ordering on the set $L=\{(i, j): 1 \leq i < j \leq \omega\}$ in the following way.
For two distinct elements $(i_1,j_1),(i_2,j_2)\in L$, we say that $(i_1,j_1)$ precedes $(i_2,j_2)$, denoted by $(i_1,j_1)<_L(i_2,j_2)$ if either $i_1<i_2$ or $i_1=i_2$ and $j_1<j_2$.
For every $(i,j)\in L$, let $C_{i,j}=\{v\in V(G)\backslash A:v\notin N(v_i)\cup N(v_j)\}\backslash\{\bigcup\limits_{(i',j')<_L(i,j)} C_{i',j'}\}$.
Note that, for any $k\in \{1,2,\ldots,j-1\}\backslash\{i\}$, $[v_k,C_{i,j}]$ is complete.
Hence $\omega(\langle C_{i,j}\rangle)\leq\omega(G)-j+2$. 

For $1\leq i\leq \omega$, let us define $I_i=\{v\in V(G)\backslash A: v\in N(a), \mathrm{for}\ \mathrm{every}\ a\in A\backslash \{v_i\}\}$.
Since $A$ is a maximum clique, for $1\leq i\leq \omega$, $I_i$ is an independent set and for any $x\in I_i$, $xv_i\notin E(G)$.
Clearly, each vertex in $V(G)\backslash A$ is non-adjacent to at least one vertex in $A$. 
Hence those vertices will be contained either in $I_i$ for some $i\in\{1,2,\ldots,\omega\}$, or in $C_{i,j}$ for some $(i,j)\in L$.
Thus $V(G)=A\cup\left(\bigcup\limits_{i=1}^{\omega}I_i\right)\cup\left(\bigcup\limits_{(i,j)\in L}C_{i,j}\right)$.


In this paper, we begin by giving alternate proofs for the $\chi$-binding functions obtained by C. Brause et al. for the class of $\{2K_2,gem\}$-free graphs, $\{2K_2,HVN\}$-free graphs and $\{2K_2, K_5-e\}$. Also, we have improved the $\chi$-binding function obtained by T. Karthick and S. Mishra for the class of $\{2K_2, K_1+C_4\}$-free graphs to $\omega(G)+1$ when $\omega\geq3$. 
In addition, we show that the class of $\{2K_2, K_2+P_4\}$-free graphs is $(\omega(G)+2)$-colorable.
Further, we prove that the families of $\{pK_2, H\}$-free graphs, where $H\in \{gem, diamond, K_2+P_4, HVN, K_5-e\}$ admit linear $\chi$-binding functions and the family of $\{pK_2, H\}$-free graphs, where $H\in \{butterfly, gem^+, dart\}$ admit quadratic $\chi$-binding functions.
In addition, we show that the families of $\{pK_2, H\}$-free graphs, where $H\in \{K_1 + C_4, C_4, \overline{P_5}\}$ admit $O(\omega^{p-1})$ $\chi$-binding functions.
Some graphs that are considered as a forbidden induced subgraphs in this paper are shown in Figure \ref{somesplgraphs}.
\begin{figure}[t]
	\centering
		\includegraphics[width=0.8\textwidth]{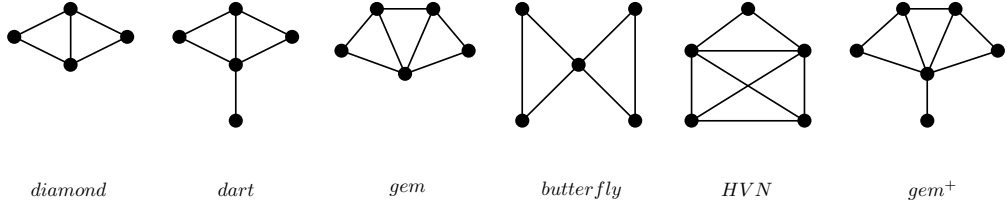}
	\caption{Some special graphs}
	\label{somesplgraphs}
\end{figure}

Notations and terminologies not mentioned here are as in \cite{west2005introduction}.
\section{$\chi$-binding functions for some $2K_2$-free graphs}\label{sect2k2}
Let us start Section \ref{sect2k2} by recalling some of the $\chi$-binding results for $2K_2$-free graphs due to S. Wagon in \cite{wagon1980bound} and S. Gaspers and S. Huang \cite{gaspers20192p_2}. 
\begin{theorem}[\cite{wagon1980bound}]\label{2k2}
If $G$ is a $2K_2$-free graph, then $\chi(G) \leq \binom{\omega(G)+1}{2}$.
\end{theorem}
\begin{theorem}[\cite{gaspers20192p_2}]\label{2k2k4}
If $G$ is a $2K_2$-free graph such that $\omega(G)\leq 3$, then $\chi(G)\leq 4$.
\end{theorem}
All graphs considered in Section 2 will be $2K_2$-free and hence throughout Section \ref{sect2k2}, for every $(i,j)\in L$, $C_{i,j}$ will always be an independent set. 
Now, let us make some simple observations of $gem$-free graphs which will be useful in finding a $\chi$-binding function for $\{2K_2,gem\}$-free graphs.
\begin{lemma}\label{gemlemma}
Let $G$ be a $gem$-free graph and $V(G)=A\cup\left(\bigcup\limits_{i=1}^{\omega}I_i\right)\cup\left(\bigcup\limits_{(i,j)\in L}C_{i,j}\right)$.  For $i,j\in\{1,2,\ldots,\omega(G)\}$ such that $i<j$ and $j\geq 3$ the following holds.

\begin{enumerate}[(i)]

\item\label{Gperfect} $\langle C_{i,j}\rangle$ is $P_4$-free and hence perfect (see \cite{seinsche1974property}).

\item\label{Gadjacency} For any $l\in\{1,2,\ldots,\omega\}$, if $H$ is a component in $\langle C_{i,j}\rangle$ and $a\in V(H)$ such that $av_l\in E(G)$, then $[V(H),v_l]$ is complete.

\item\label{GIp} For $a\in C_{i,j}$ if $av_l\notin E(G)$, for some $l\in\{1,2,\ldots,\omega\}$, then $[a,I_l]=\emptyset$.

\item\label{Gclique} If $H$ is a component of $C_{i,j}$, then $\omega(H)\leq |A\backslash N_A (V(H))|$.
\end{enumerate}
\end{lemma}
\begin{proof}
\begin{enumerate}[(i)]

\item  Suppose there exists a $P_4\sqsubseteq \langle C_{i,j}\rangle$, say $P$.
Since $j\geq 3$, there exists a $s\in \{1,2,\ldots,j\}\backslash\{i,j\}$ such that $N(v_s)\supseteq\ V(P)$ and hence $\langle V(P)\cup\{v_s\}\rangle\cong gem$, a contradiction.

\item On the contrary, let $H$ be a component of $\langle C_{i,j}\rangle$ and $a,b\in V(H)$ such that $ab, av_l\in E(G)$ but $bv_l\notin E(G)$, for some $l\in \{1,2,\ldots,\omega\}$. Clearly, $l\in\{j+1,j+2,\ldots,\omega\}$.
Since $j\geq3$, there exists an integer $s\in \{1,2,\ldots,j\}\backslash\{i,j\}$ such that $N(v_s)\supseteq\{a,b,v_i,v_j,v_l\}$ and thus $\langle\{v_s,b,a,v_l,v_j\}\rangle\cong gem$, a contradiction.

\item Let us suppose  that there exist vertices $a\in C_{i,j}$ and $b\in I_l$ such that $av_l\notin E(G)$ and $ab\in E(G)$.
Let $k\in\{i,j\}\backslash\{l\}$.
Since $j\geq 3$, there exists an integer $s\in \{1,2,\ldots,j\}\backslash\{i,j\}$ such that $N(v_s)\supseteq\{a,b,v_k,v_l\}$ and hence $\langle\{v_s, a,b,v_k,v_l\}\rangle\cong gem$, a contradiction.

\item Follows immediately from (\ref{Gadjacency}). 
\end{enumerate}
\end{proof}
Let us now consider $\{2K_2, gem\}$-free graphs.
Theorem \ref{3oromega} has already been proved by C. Brause et al. in \cite{brause2019chromatic}.
We have given an alternative proof which is much simpler compared to the one given in \cite{brause2019chromatic}.
\begin{theorem}\label{3oromega}
If $G$ is a $\{2K_2, gem\}$-free graph, then $\chi(G) \leq \max \{ 3 ,\omega(G) \}$.
\end{theorem}
\begin{proof}
Let $G$ be a $\{2K_2, gem\}$-free graph and $V(G)=A\cup\left(\bigcup\limits_{i=1}^{\omega}I_i\right)\cup\left(\bigcup\limits_{(i,j)\in L}C_{i,j}\right)$. 
If $\omega(G)=2$, by using Theorem \ref{2k2}, the result follows.
Now, let us consider $\omega(G)\geq 3$.
Since $G$ is $gem$-free, for any $i,j\in\{1,2,\ldots,\omega\}$, $[I_i,I_j]$ is complete
(Otherwise there will exist vertices $a\in I_i$ and $b\in I_j$ such that $ab\notin E(G)$. Since $\omega(G)\geq 3$, we can find $p\in \{1,2,\ldots,\omega\}\backslash\{i,j\}$ such that $\langle\{v_p,a,v_j,v_i,b\}\rangle\cong gem$, a contradiction).
For $j\geq 4$, we claim that $\cup_{i=1}^{j-1} C_{i,j}$ is an independent set.
If there exist vertices $a\in C_{i,j}$ and $b\in C_{k,j}$ such that $ab\in E(G)$, $i,k\in\{1,2,\ldots,j-1\}$ and $i\neq k$, then we can find an integer $s\in \{1,2,\ldots,j\}\backslash \{i,k,j\}$ such that $av_s,bv_s\in E(G)$ and hence $\langle\{v_s,a,b,v_i,v_j\}\rangle\cong gem$, a contradiction.
Next, let us consider $C_{1,2}$.
If $a\in C_{1,2}$ such that $[a,I_1]\neq \emptyset$, then we shall show that for $p\neq1$, $[a,\{v_p\}\cup I_p]=\emptyset$ and $[a,C_{i,j}]=\emptyset$ for every $(i,j)\in L\backslash \{(1,2)\}$, that is, $N(a)\subseteq I_1$.
Let $a\in C_{1,2}$ and $b\in I_1$ such that $ab\in E(G)$.
If $s\in \{1,2\ldots,\omega\}\backslash\{1,2\}$ such that $av_s\in E(G)$, then $\langle\{v_s,a,b,v_2,v_1\}\rangle\cong gem$, a contradiction. 
In a similar fashion we can show that, for $p\neq 1$, $[a,I_p]=\emptyset$.
Finally, if there exist vertices $a\in C_{1,2}$ and $c\in C_{i,j}$ such that $ac\in E(G)$, then by using the fact that $[a,A]=\emptyset$, $\langle\{a,c, v_i, v_j\}\rangle\cong 2K_2$, a contradiction.

Let us next establish an $\omega$-coloring for $G$ using the colors  $\{1,2,\ldots,\omega\}$.
For $1\leq i\leq \omega$, let us give the color $i$ to the vertices in $\{v_i\}\cup I_i$.
By using (\ref{GIp}) of Lemma \ref{gemlemma}, for $j\geq 4$, we can assign the color $j$ to the vertices in $\cup_{i=1}^{j-1} C_{i,j}$ and assign the colors $3$ and $2$ to the vertices in $C_{1,3}$ and $C_{2,3}$ respectively.
Finally for $C_{1,2}$, if $a\in C_{1,2}$ such that $[a,I_1]=\emptyset$, then we can assign the color $1$ to $a$. If not, $N(a)\subseteq I_1$ and hence we can assign the color $2$.
Clearly this is a proper coloring of $G$ and thus $\chi(G)\leq \omega(G)$.
\end{proof}
Now without much difficulty, one can observe that the properties of $gem$-free graphs mentioned in Lemma \ref{gemlemma} will hold for $\{K_2+P_4\}$-free graphs. 
\begin{lemma}\label{pk2k2+p4lemma}
Let $G$ be a $\{K_2+P_4\}$-free graph and $V(G)=A\cup\left(\bigcup\limits_{i=1}^{\omega}I_i\right)\cup\left(\bigcup\limits_{(i,j)\in L}C_{i,j}\right)$.
For $i,j\in\{1,2,\ldots,\omega(G)\}$ such that $i<j$ and $j\geq 4$ the following holds.

\begin{enumerate}[(i)]

\item\label{Gperfect1}  $\langle C_{i,j}\rangle$ is $P_4$-free and hence it is perfect.
\item\label{Gadjacency1}  For $j<l$, let $H$ be a component in $\langle C_{i,j}\rangle$ and let $a\in V(H)$ such that $av_l\in E(G)$, then $[V(H),v_l]$ is complete.
\item\label{GIp1} For $l\leq\omega$, let $a\in C_{i,j}$ such that $av_l\notin E(G)$, then $[a,I_l]=\emptyset$.
\item\label{Gclique1} If $H$ is a component of $C_{i,j}$, then $\omega(H)\leq |A\backslash N_A (V(H))|$.
\end{enumerate}
\end{lemma}
\begin{proof}
Proof follows by similar arguments as given in Lemma \ref{gemlemma}.
\end{proof}
Also, it can be seen that many of the properties of $\{2K_2,gem\}$-free graphs with clique size at least $3$ which is mentioned in Theorem \ref{3oromega} will even hold for $\{2K_2,K_2+P_4\}$-free graphs with clique size at least $4$.
Now, let us establish a $\chi$-binding function for $\{2K_2,K_2+P_4\}$-free graphs.
\begin{theorem}\label{2k2k2+p4}
If $G$ is a $\{2K_2, K_2+P_4\}$-free graph with $\omega(G)\geq 4$, then $\chi(G)\leq\omega(G) + 2$.
\end{theorem}
\begin{proof}
Let $G$ be a $\{2K_2,K_2+P_4\}$-free graph with $\omega(G)\geq 4$.
As done in Theorem \ref{3oromega}, by using similar arguments, we can observe that for $1\leq i,j\leq \omega(G)$, $[I_i,I_j]=complete$ and for $j\geq 5$, $\bigcup_{i=1}^{j-1} C_{i,j}$ is an independent set.
Further we can also show that, if there exists a vertex $a\in C_{2,3}$ such that $[a,I_3]\neq \emptyset$, then for any $p\geq 4$, $[a,I_p]=\emptyset$, $[a,A\backslash\{v_1\}]=\emptyset$ and $[a,C_{3,4}]=\emptyset$. Let $a\in C_{2,3}$, $b\in I_3$ such that $ab\in E(G)$.
Suppose $[a,I_p]\neq\emptyset$, then there exists $c\in I_p$ for $p\geq 4$ such that $ac\in E(G)$. Hence, $\langle\{v_1,c,a,b,v_2,v_3\}\rangle\cong K_2+P_4$, a contradiction.
Similarly, we can also show that $[a,A\backslash\{v_1\}]=\emptyset$.
Suppose $d\in C_{3,4}$ such that $ad\in E(G)$, then by using the fact that $[a,A\backslash\{v_1\}]=\emptyset$, we get that $\langle\{a,d,v_3,v_4\}\rangle\cong 2K_2$, a contradiction.

Now let us give an $(\omega+2)$-coloring for $G$ using the colors $\{1,2,\ldots,\omega+2\}$.
For $1\leq i\leq \omega$, let us assign the color $i$ to the vertices in$\{v_i\}\cup I_i$.
For $j\geq 5$, assign the color $j$ to the vertices in$\cup_{i=1}^{j-1}C_{i,j}$ and assign the colors $1$, $2$, $4$, $\omega+1$ and $\omega+2$ to the vertices in $C_{1,4}$, $C_{2,4}$, $C_{3,4}$, $C_{1,2}$ and $C_{1,3}$ respectively.
For a vertex $a\in C_{2,3}$ if $[a,I_3]=\emptyset$, then assign the color $3$ to the vertices in $C_{2,3}$ or else assign  the color $4$.
Clearly this would be a proper coloring for $G$ using $\omega(G)+2$ colors.
\end{proof}
In \cite{karthick2018chromatic}, T. Karthick and S. Mishra have shown that if $G$ is a $\{2K_2, HVN\}$-free graph, then $\chi(G)\leq \omega(G) + 3$.
In addition, C. Brause et al. in \cite{brause2019chromatic} have already proved that if $G$ is a $\{2K_2, HVN\}$-free graph with $\omega\geq4$, then $G$ is a perfect. 
Theorem \ref{2k2hvn} becomes a corollary to this but the proof given by us is much simpler compared to the one given in \cite{brause2019chromatic}. 
\begin{theorem}\label{2k2hvn}
If $G$ is a $\{2K_2, HVN\}$-free graph, such that $\omega\geq 4$, then $\chi(G)=\omega(G)$.
\end{theorem}
\begin{proof}
Let $G$ be a $\{2K_2, HVN\}$-free graph with $\omega(G)\geq 4$.
Since $G$ is $HVN$-free and $\omega(G)\geq 4$, for any $i,j\in\{1,2,\ldots,\omega\}$, $[I_i,I_j]$ is complete  
(Otherwise there will exist vertices $a\in I_i$ and $b\in I_j$ such that $ab\notin E(G)$. Since $\omega(G)\geq 4$, we can find $p,q\in \{1,2,\ldots,\omega\}\backslash\{i,j\}$ such that $\langle\{a,b,v_j,v_p,v_q\}\rangle\cong HVN$, a contradiction).
Now, for $j\geq4$, we can observe that $C_{i,j}=\emptyset$.
Suppose there exists a vertex $a\in C_{i,j}$, then we can find $s,q\in \{1,2,\ldots,j\}\backslash\{i,j\}$ such that $av_s,av_q\in E(G)$ and $\langle \{a,v_s,v_q,v_i,v_j\}\rangle\cong HVN$, a contradiction.
Thus $V(G)= A \cup C_{1,2}\cup C_{1,3}\cup C_{2,3}\cup\left(\bigcup\limits_{i=1}^{\omega}I_i\right)$. Let $B=A\cup\left(\bigcup\limits_{i=1}^{\omega}I_i\right)$.
By using similar arguments, we can also observe that $N_A (C_{1,3})=\{v_2\}$, $N_A (C_{2,3})=\{v_1\}$ and if a vertex $a\in C_{1,2}$ then $\left|[a,A]\right|\leq 1$. 
In addition, $N_B (C_{1,3})\subseteq(\{v_2\}\cup I_2)$. 
Suppose $N_B (C_{1,3})\not\subseteq (\{v_2\}\cup I_2)$, then there exists $a\in C_{1,3}$ and $b\in I_k$, for some $k\neq 2$, such that $ab\in E(G)$.
Since $\omega(G)\geq 4$ and $N_A (C_{1,3})=\{v_2\}$, there exist $s,q\in \{1,2,\ldots,\omega\}\backslash\{2,k\}$ such that $av_s,av_q\notin E(G)$ and thus $\langle \{a,v_2,b,v_s,v_q\}\rangle\cong HVN$, a contradiction.
Similar arguments will show that $N_B (C_{2,3})\subseteq(\{v_1\}\cup I_1)$ and if $a\in C_{1,2}$, then for some $s\in \{1,2,\ldots,\omega\}$, $N_B (C_{1,2})\subseteq(\{v_s\}\cup I_s)$.

Let us now exhibit an $\omega$-coloring for $G$ using $\{1,2,\ldots,\omega\}$ colors.
For $1\leq i\leq \omega$, let us give the color $i$ to the  the vertices in$\{v_i\}\cup I_i$ and the colors $1$ and $2$ to the vertices in $C_{1,3}$ and $C_{2,3}$ respectively. Finally, for the vertices in $C_{1,2}$, if $a\in C_{1,2}$ such that $N_B (a)\subseteq(\{v_3\}\cup I_3)$ then color $a$ with $4$ or else color it with $3$.
Clearly, this is a proper coloring of $G$ and thus $\chi(G)\leq\omega(G)$.
\end{proof}

%
In \cite{karthick2018chromatic}, T. Karthick and S. Mishra  have shown that if $G$ is a $\{2K_2, K_5-e\}$-free graph, then $\chi(G)\leq \omega(G) + 4$.
In Theorem \ref{2k2k5-e}, we have improved the bound for $\omega(G)=4$ and $5$.
Also C. Brause et al. in \cite{brause2019chromatic} have proved that if $G$ is a $\{2K_2, K_5-e\}$-free graph with $\omega(G)\geq6$, then $G$ is a perfect. 
For $\omega\geq6$, even though Theorem \ref{2k2k5-e} is a simple consequence of what C. Brause et al. have proved, we have provided an alternative proof which again is much simpler to the one given in \cite{brause2019chromatic}.
\begin{theorem}\label{2k2k5-e}
If $G$ is a $\{2K_2, K_5-e\}$-free graph with $\omega(G)\geq 4$, then 
$\chi(G)\leq \left\{
\begin{array}{lcl}
6 					& \mbox{for} & \omega(G)= 4	\\ 
\omega(G)  	& \mbox{for} & \omega(G)\geq 5.
\end{array}
\right.$
\end{theorem}
\begin{proof}
Let $G$ be a $\{2K_2, K_5-e\}$-free graph with $\omega(G)\geq4$.
Let us start by observing that, for $i\in\{1,2,\ldots,\omega\}$, $I_i=\emptyset$.
Suppose there exists an $i\in\{1,2,\ldots,\omega\}$ such that $I_i\neq\emptyset$, say $a\in I_i$, then we can find $p,q,r\in\{1,2\ldots,\omega\}\backslash \{i\}$ such that $\langle \{a,v_p,v_q,v_r,v_i\}\rangle\cong K_5-e$, a contradiction.
For $\omega=4$, let us partition the vertices into six independent sets, namely,
$V_1=v_1\cup  C_{1,2}$, $V_2=v_2\cup  C_{2,3}$, $V_3=v_3\cup  C_{1,3}$, $V_4=v_4\cup  C_{1,4}$, $V_5=C_{2,4}$ and $V_6=C_{3,4}$.
Hence $\chi(G)=6$, when $\omega=4$.
Let us next consider $\omega\geq5$.
Without much difficulty, we can see that $C_{i,j}=\emptyset$, for any $j\geq5$ 
(Otherwise, for some $j\geq5$, there exists a vertex $a\in C_{i,j}$ and we can find $p,q,r\in\{1,2,\ldots,j\}\backslash\{i,j\}$ such that $\langle \{a,v_p,v_q,v_r,v_i\}\rangle\cong K_5-e$, a contradiction).

Thus $V(G)=A\cup C_{1,2}\cup C_{1,3}\cup C_{2,3}\cup C_{1,4}\cup C_{2,4}\cup C_{3,4}$.
For $1\leq i\leq3$, we shall show that  $N_A(C_{i,4})=\{v_1,v_2,v_3\}\backslash\{v_i\}$. Suppose there exists an integer $j\geq5$ and some vertex $a\in C_{i,4}$ such that $av_j\in E(G)$,
then $\langle \{a,v_1,v_2,v_3,v_j\}\rangle\cong K_5-e$, a contradiction.
Moreover, $\cup_{i=1}^3 C_{i,4}$ is independent.
On the contrary, if there exist vertices $a\in C_{i,4}$, $b\in C_{k,4}$ (for some $1\leq i,k\leq3$) such that $ab\in E(G)$,
then $\langle \{a,b,v_4,v_5\}\rangle\cong 2K_2$, a contradiction.
Now, let us partition the graph into $\omega$ independent sets in the following way.
$V_1=v_1\cup  C_{1,2}$, $V_2=v_2\cup  C_{2,3}$, $V_3=v_3\cup  C_{1,3}$, $V_4=v_4\cup C_{1,4}\cup C_{2,4}\cup C_{3,4}$, and for $5\leq i\leq\omega$,  $V_i=\{v_i\}$. Hence $\chi(G)=\omega(G)$.
\end{proof}
In \cite{karthick2018chromatic}, T. Karthick and S. Mishra  have showed that if $G$ is a $\{2K_2, K_1 + C_4\}$-free graph, then $\chi(G)\leq \omega(G) + 5$.
In Theorem \ref{2k2k1+c4}, we improve the bound to $\omega(G)+1$.
\begin{theorem}\label{2k2k1+c4}
If $G$ is a $\{2K_2, K_1 + C_4\}$-free graph with $\omega(G)\geq 3$, then $\chi(G) \leq \omega(G)+1$.
\end{theorem}
\begin{proof}
Let $G$ be a $\{2K_2, K_1 + C_4\}$-free graph with $\omega(G)\geq 3$ and $V(G)=A\cup\left(\bigcup\limits_{i=1}^{\omega}I_i\right)\cup\left(\bigcup\limits_{(i,j)\in L}C_{i,j}\right)$.
Let  $\{1,2,\ldots,\omega(G)+1\}$ be the set of colors.
For $1\leq i\leq \omega$, assign the color $i$ to the vertex $v_i$.
For $1\leq i,j\leq \omega$, one can easily observe that  $[I_i, I_j]=\emptyset$.
If not, for some $i,j\in\{1,2,\ldots,\omega\}$, if there exist vertices $a\in I_i$ and $b\in I_j$ such that $ab\in E(G)$,
then we can find an $s\neq \{i,j\}$ such that $\langle\{v_s,a,v_j,v_i,b\}\rangle\cong K_1 + C_4$, a contradiction.
Hence $\langle\cup_{i=1}^\omega I_i\rangle$ is independent and can be given a single color. Let it be  $\omega+1$.
Next, for $j\geq 4$, we shall show that $[C_{i,j}, C_{k,j}]=\emptyset$.
Suppose there exist vertices $a\in C_{i,j}$ and $b\in C_{k,j}$ such that $ab\in E(G)$, then we can find $s\in \{1,2,\ldots,j\}\backslash\{i,k,j\}$ such that $\langle \{v_s,a,v_k,v_i,b\}\rangle\cong K_1 + C_4$, a contradiction.
Hence for all $j\geq 4 $, $\langle\cup_{i=1}^{j-1} C_{i,j}\rangle$ is independent and $[v_j,\cup_{i=1}^{j-1} C_{i,j}]=\emptyset$.
Thus for any $j\geq4$, we can assign the color $j$ to all the vertices in$\cup_{i=1}^{j-1}C_{i,j}$.
Also the colors $1,3$ and $2$ can be given to the vertices in $C_{1,2}$, $C_{1,3}$ and $C_{2,3}$ respectively.
Hence $G$ is $(\omega+1)$-colorable.
%
\end{proof}
Let us recall a graph transformation defined by J. Mycielski \cite{mycielski1955coloriage} as follows. 
For a graph $G=(V,E)$, the Mycielskian of $G$, denoted by $\mu(G)$, is the graph with vertex set $V(\mu(G))=V\cup V'\cup \{u\}$ where $V'=\{x':x\in V\}$ and the edge set $E(\ma)=E\cup\{xy':xy\in E\}\cup\{y'u:y'\in V'\}$. 
Without much difficulty, it can be seen that $\mu(K_\omega)$ will be $\{2K_2, K_1 + C_4\}$-free with chromatic number equal to $\omega+1$
and hence the bound given in Theorem \ref{2k2k1+c4} is tight.
As a consequence of Theorem \ref{2k2k1+c4}, we will get Corollary \ref{2k2c4} which was proved by Z. Bl\'{a}zsik et al. in \cite{blazsik1993graphs}.
\begin{corollary}\label{2k2c4}\cite{blazsik1993graphs}
If $G$ is a $\{2K_2, C_4\}$-free graph, then $\chi(G)\leq \omega(G)+1$.
\end{corollary}
\section{$\chi$-binding functions for some $pK_2$-free graphs}\label{sectpk2}
In Section \ref{sectpk2}, let us find $\chi$-binding function for some classes of $pK_2$-free graphs. 
Let us begin by recalling a result due to S. Brandt in \cite{brandt2002triangle}.
\begin{theorem}[\cite{brandt2002triangle}]\label{pk2}
For $p\geq3$, if $G$ is a $pK_2$-free graph with $\omega(G)=2$, then $\chi(G)\leq 2p-2$.
\end{theorem}
For $p\geq 3$, let us consider $pK_2$-free graphs with $\omega(G)=3$.
\begin{theorem}\label{pk2omega3}
For $p\geq3$, if $G$ is a $pK_2$-free graph with $\omega(G)=3$, then $\chi(G)\leq 2p^2-3p+4$.
\end{theorem}
\begin{proof}
Let $p\geq 3$ and $G$ be a $pK_2$-free graph with $\omega(G)=3$. Then $V(G)= A\cup\left(\bigcup\limits_{i=1}^{3}I_i\right)\cup(C_{1,2}\cup C_{1,3}\cup C_{2,3})$ and each of $C_{1,2}$, $C_{1,3}$ and $C_{2,3}$ are $(p-1)K_2$-free.
Also, by the definition of $C_{i,j}$, $\omega(C_{1,2})\leq 3$, $\omega(C_{1,3})\leq 2$ and $\omega(C_{2,3})\leq 2$.
Let us prove the result by induction on $p\geq 3$.

Let us first give the color $i$ to all the vertices in $\{v_i\}\cup I_i$, where $i\in\{1,2,3\}$.
For $p=3$, by using Theorem \ref{2k2} and Theorem \ref{2k2k4}, $C_{1,3}$ and $C_{2,3}$ can be colored with $3$ colors each and $C_{1,2}$ can be colored with $4$ colors.
Therefore, on the whole $G$ can be colored with at most $3+2\times3+4=13$ colors.
For $s\geq3$,  let us assume that the result is true for any $sK_2$-free graph with $\omega=3$.

Let us consider $G'$ to be an $(s+1)K_2$-free graph with $\omega(G')=3$.
By using Theorem \ref{pk2}, $C_{1,3}$ and $C_{2,3}$ can be colored with $2s-2$ colors each and by our assumption $C_{1,2}$ can be colored with at most $2s^2-3s+4$ colors.
Therefore, altogether $G'$ can be colored with $3+2(2s-2)+2s^2-3s+4= 2(s+1)^2-3(s+1)+4$ colors.
%
%
%
\end{proof}
The strategy of the proofs in Section \ref{sectpk2} are all the same. 
So instead of repeating the common steps for each proof, we have mentioned 2 strategies which will include the portion of the proof that will be common.
 
\noindent\textbf{Strategy 1}

Let $G$ be a $\{pK_2,H\}$-free graph where $H$ is some graph. 
While considering the case when $\omega=1$, the result will be trivial and for the case when $\omega=2$, the results will follow from Theorem \ref{2k2} and Theorem \ref{pk2}.
Now for $\omega\geq 3$, we shall always prove the result by induction on $p$.
For $p=2$, we shall show by using the result corresponding to the case when $p=2$ from Section \ref{sect2k2}. 
For $s\geq 2$, we shall always  assume that the result is true for any $\{sK_2, H\}$-free graph and prove that the result is true for $\{(s+1)K_2, H\}$-free graph. 

\noindent\textbf{Strategy 2}

Let $G$ be a $\{pK_2,H\}$-free graph where $H$ is some graph. 
While considering the case when $\omega=1$, the result will be trivial and for the case when $\omega=2$,  the results will follow from Theorem \ref{2k2} and Theorem \ref{pk2} and for the case when $\omega=3$, the results will follow from Theorem \ref{2k2k4} and Theorem \ref{pk2omega3}.
Now for $\omega\geq 4$, we shall prove the result by induction on $p$.
For $p=2$, we shall show by using the result corresponding to the case when $p=2$ from Section \ref{sect2k2}. 
For $s\geq 2$, we shall  assume that the result is true for any $\{sK_2, H\}$-free graph and prove that the result is true for $\{(s+1)K_2, H\}$-free graph. 

So, in all the proofs of the Theorems in Subsections \ref{subsectlinear} and \ref{subsecthigher}, we shall apply either Strategy 1 or Strategy 2 and only show that the result is true for $\{(s+1)K_2, H\}$-free graph.


Let us now divide Section \ref{sectpk2} into subsections depending upon the order of the $\chi$-binding function.
\subsection{Linear $\chi$-binding functions}\label{subsectlinear}

In Subsection \ref{subsectlinear}, we will consider classes of $pK_2$-free graphs for which we can establish a linear $\chi$-binding function.
Let us start Subsection \ref{subsectlinear} by considering the $\chi$-binding function for $\{pK_2, gem\}$-free graphs. 
For $p\geq2$, let us define a function $f^{\thecount}_p:\mathbb{N}\rightarrow\mathbb{N}$ 
as follows.
For $t\geq2$ and $p,m\geq3$, $f^{\thecount}_t(1)=1$, $f^{\thecount}_2(2)=3$, $f^{\thecount}_2(m)=m$, $f^{\thecount}_p(2)=2p-2$ and $f^{\thecount}_p(m)=f^{\thecount}_{p-1}(m)+2m-2$. 
Without much difficulty, one can observe that $f_p^{\thecount}$ is a linear polynomial.
\begin{theorem}\label{pk2gem}
For $p\geq2$, if $G$ is a $\{pK_2, gem\}$-free graph, then $\chi(G)\leq f^{\thecount}_p(\omega(G))$.
\end{theorem}
\begin{proof}

Let us apply Strategy 1.
Let $G$ be an $\{(s+1)K_2, gem\}$-free graph and
 $V(G)=A\cup\left(\bigcup\limits_{i=1}^{\omega}I_i\right)\cup\left(\bigcup\limits_{(i,j)\in L}C_{i,j}\right)$.
For $1\leq i\leq \omega$, let us assign the color $i$ to the  vertices in$\{v_i\}\cup I_i$.
For $j\geq 3$, if $H$ is a component of $\langle C_{i,j}\rangle$, then by using (\ref{Gperfect}) of Lemma \ref{gemlemma}, $\chi(H)=\omega(H)$.
Also, for $i\neq 2$ and $j\geq 3$, by using (\ref{GIp}) and (\ref{Gclique}) of Lemma \ref{gemlemma}, every component of $\langle C_{i,j}\rangle$ can be colored using the colors given to the vertices in $A$.
We shall show that this will be a proper coloring.
Suppose there exists two adjacent vertices $a\in C_{i,j}$ and $b\in C_{k,l}$ which receives the same color, say $q$.
Without loss of generality, assume $(i,j)<_L(k,l)$, then either $i<k$ or $i=k$ and $j<l$.
If $i<k$, then $bv_i\in E(G)$.
Since the vertices $a$ and $b$ have been assigned the same color $q$, we have $av_q, bv_q\notin E(G)$.
Also, for $i\neq 2,k\neq 2$ and $j,l\geq 3$, $N(v_2)\supseteq\{a,b,v_i,v_q\}$ and thus $\langle\{v_2, a,b,v_i,v_q\}\rangle\cong gem$, a contradiction.
If $i=k$ and $j<l$, then $bv_j\in E(G)$ and by similar arguments, we will get a contradiction.

For any $l\geq3$, we know that $\omega(\langle C_{2,l}\rangle)\leq\omega(G)-1$.  
For $i\neq 2$ and $j\geq 3$, we have $[v_2,C_{i,j}]$ is complete.
By using (\ref{GIp}) of Lemma \ref{gemlemma}, the color $2$ is available for all the vertices in$\cup_{l=3}^\omega C_{2,l}$.
For any $r>s\geq3$, we claim that  $[C_{2,r},C_{2,s}]=\emptyset$.
If there exist vertices $a\in C_{2,r}$ and $b\in C_{2,s}$ such that $ab\in E(G)$, then $\langle\{v_1,b,a,v_s,v_2\}\rangle\cong gem$, a contradiction.
By using (\ref{Gperfect}) of Lemma \ref{gemlemma}, we can color all the vertices in$\cup_{l=3}^\omega C_{2,l}$ with the color $2$ together with at most $\omega-2$ new colors.

Thus the vertices in$V(G)\backslash C_{1,2}$ can be colored using $2\omega-2$ colors.
Since $\langle C_{1,2}\rangle$ is $\{sK_2, gem\}$-free, by induction hypothesis, it can be colored with $f^{\thecount}_s(\omega)$ colors.
Therefore, $V(G)$ can be colored with at most $f^{\thecount}_s(\omega)+2\omega-2=f^{\thecount}_{s+1}(\omega)$ colors.
\end{proof}
Since $diamond$ is an induced subgraph of $gem$, for $p\geq2$, $f^{\thecount}_p(\omega(G))$ will be the $\chi$-binding function for $\{pK_2, diamond\}$-free graphs as well.
\begin{corollary}\label{pk2diamond}
For $p\geq2$, if $G$ is a $\{pK_2, diamond\}$-free graph, then $\chi(G)\leq f^{\thecount}_p(\omega(G))$.
\end{corollary}
\newcounter{diamond}
\setcounter{diamond}{\thecount}
\stepcounter{count}
Next, for $p\geq2$, let us define a function $f^{\thecount}_p:\mathbb{N}\rightarrow\mathbb{N}$ (which will serve as a $\chi$-binding function for $\{pK_2, K_2+P_4\}$-free graphs) as follows.
For $t\geq2$, $m,p\geq3$ and $s\geq 4$, define $f^{\thecount}_t(1)=1$, $f^{\thecount}_2(2)=3$, $f^{\thecount}_2(3)=4$, $f^{\thecount}_2(s)=s+2$, $f^{\thecount}_p(2)=2p-2$, $f^{\thecount}_p(3)=2p^2-3p+4$ and $f^{\thecount}_p(m)=f^{\thecount}_{p-1}(m)+2f^{\thediamond}_{p-1}(m-1)+3m-6$.
Without much difficulty, one can observe that $f_p^{\thecount}$ is linear.
\begin{theorem}\label{pk2k2+p4}
For $p\geq 2$, if $G$ is a $\{pK_2,K_2+P_4\}$-free graphs, then $\chi(G)\leq f^{\thecount}_p(\omega(G))$.
\end{theorem}
\begin{proof}
Let us apply Strategy 2.
Let $G$ be an $\{(s+1)K_2, K_2+P_4\}$-free graph and  $V(G)=A\cup\left(\bigcup\limits_{i=1}^{\omega}I_i\right)\cup\left(\bigcup\limits_{(i,j)\in L}C_{i,j}\right)$.
As done in Section \ref{sect2k2}, for $\{2K_2,gem\}$-free graphs  and $\{2K_2,K_2+P_4\}$-free graphs, here also
by using similar coloring technique as done in Theorem \ref{pk2gem}, we can color the vertices in $A\cup\left(\bigcup\limits_{i=1}^{\omega}I_i\right)\cup\left(\bigcup\limits_{i\geq3}C_{i,j}\right)$ with at most $\omega$ colors.
Also, as in Theorem \ref{pk2gem}, by using (\ref{Gperfect}) of Lemma \ref{pk2k2+p4lemma}, we can color all the vertices in$\cup_{l=4}^\omega C_{2,l}$ and $\cup_{l=4}^\omega C_{1,l}$ with $\omega-2$ colors each such that the color $2$ and color $1$ respectively can be again used.

Thus the vertices in$V(G)\backslash (C_{1,2}\cup C_{1,3}\cup C_{2,3})$ have been colored using $3\omega-6$ colors.
Clearly $\langle C_{i,j}\rangle$ is $\{sK_2,K_2+P_4\}$-free for every $(i,j)\in L$.
Since $[v_2,C_{1,3}]=complete$ and $[v_1,C_{2,3}]=complete$, both $\langle C_{1,3}\rangle$ and $\langle C_{2,3}\rangle$ are $\{sK_2,gem\}$-free and hence can be colored with $f^{\thediamond}_{s}(\omega-1)$ colors each.
Also, $\langle C_{1,2}\rangle$ is $\{sK_2, K_2+P_4\}$-free, and hence it can be colored with $f^{\thecount}_s(\omega)$ colors.
Therefore, altogether $G$ can be colored with at most $f^{\thecount}_s(\omega)+2f^{\thediamond}_{s}(\omega-1)+3\omega-6=f^{\thecount}_{s+1}(\omega)$ colors.
\end{proof}
\stepcounter{count}
Next, for $p\geq2$, let us define a function $f^{\thecount}_p:\mathbb{N}\rightarrow\mathbb{N}$ which will serve as a $\chi$-binding function for $\{pK_2, HVN\}$-free graphs as follows.
For $t\geq2$, $p\geq3$ and $m,s\geq4$, $f^{\thecount}_t(1)=1$, $f^{\thecount}_2(2)=3$, $f^{\thecount}_2(3)=4$, $f^{\thecount}_p(2)=2p-2$, $f^{\thecount}_2(s)=s$, $f^{\thecount}_p(3)=2p^2-3p+4$ and $f^{\thecount}_p(m)=f^{\thecount}_{p-1}(m)+2f^{\thecount}_{p-1}(m-1)$. 
Here also one can observe that $f_p^{\thecount}$ is a linear polynomial on $\omega$.
\begin{theorem}\label{pk2hvn}
For $p\geq2$, if $G$ is a $\{pK_2, HVN\}$-free graph, then $\chi(G)\leq f^{\thecount}_p(\omega(G))$.
\end{theorem}
\begin{proof}

We shall apply Strategy 2 to prove the result. 
Let $G$ be an $\{(s+1)K_2, HVN\}$-free graph with $\omega\geq 4$. As observed in the proof of Theorem \ref{2k2hvn}, since $G$ is $\{HVN\}$-free, 
we have $V(G)= A \cup C_{1,2}\cup C_{1,3}\cup C_{2,3}\cup\left(\bigcup\limits_{i=1}^{\omega}I_i\right)$ and for $B=A \cup\left(\bigcup\limits_{i=1}^{\omega}I_i\right)$, $N_B (C_{1,3})\subseteq(\{v_2\}\cup I_2)$, $N_B (C_{2,3})\subseteq(\{v_1\}\cup I_1)$, if $a\in C_{1,2}$, then for some $s\in \{1,2,\ldots,\omega\}$, $N_B (C_{1,2})\subseteq(\{v_s\}\cup I_s)$. 
For $1\leq i\leq \omega$, let us give the color $i$ to the  vertices in$\{v_i\}\cup I_i$.
Hence $A\cup \left(\bigcup\limits_{i=1}^{\omega}I_i\right)$ can be colored with $\omega(G)$ colors.
Clearly, for $(i,j)\in L$,  each $\langle C_{i,j}\rangle$ is $\{sK_2, HVN\}$-free and $\omega(\langle C_{1,3}\rangle)\leq\omega(G)-1$ and $\omega(\langle C_{2,3}\rangle)\leq\omega(G)-1$.  


%


Since $N_B (C_{1,3})\subseteq(\{v_2\}\cup I_2)$, $N_B (C_{2,3})\subseteq(\{v_1\}\cup I_1)$
we can use the colors  $\{1,3,4,\ldots,\omega\}$ for the vertices in $C_{1,3}$ and use the colors $\{2,3,4,\ldots,\omega\}$ for the vertices in $C_{2,3}$.
Since $C_{i,j}$ is $\{sK_2, HVN\}$-free, the vertices in $C_{1,2}$, $C_{1,3}$ and $ C_{2,3}$ can be colored with at most $f^{\thecount}_s(\omega)$, $f^{\thecount}_s(\omega-1)$ and $f^{\thecount}_s(\omega-1)$ colors respectivley. 
Therefore, altogether $G$ can be colored with at most $\omega+f^{\thecount}_s(\omega)+2f^{\thecount}_s(\omega-1)-\omega=f^{\thecount}_s(\omega)+2f^{\thecount}_s(\omega-1)=f^{\thecount}_{s+1}(\omega)$ colors.
\end{proof}
\newcounter{hvn}
\setcounter{hvn}{\thecount}
\stepcounter{count}
Next, for $p\geq2$, let us define a function $f^{\thecount}_p:\mathbb{N}\rightarrow\mathbb{N}$ which will serve as a $\chi$-binding function for $\{pK_2, K_5-e\}$-free graphs as follows.
For $t\geq2$, $p\geq3$, $m\geq4$ and $s\geq5$, $f^{\thecount}_t(1)=1$, $f^{\thecount}_2(2)=3$, $f^{\thecount}_2(3)=4$, $f^{\thecount}_2(4)=6$, $f^{\thecount}_2(s)=s$, $f^{\thecount}_p(2)=2p-2$, $f^{\thecount}_p(3)=2p^2-3p+4$ and $f^{\thecount}_p(m)=f^{\thecount}_{p-1}(m)+2f^{\thediamond}_{p-1}(m-1)+3m-6$.
Again it can be seen that $f_p^{\thecount}$ is a linear polynomial on $\omega$.
\begin{theorem}\label{pk2k5-e}
For $p\geq2$, if $G$ is a $\{pK_2, K_5-e\}$-free graph, then $\chi(G) \leq f^{\thecount}_p(\omega(G))$.
\end{theorem}
\begin{proof}
Let us prove by applying Strategy 2.
Let $G$ be an $\{(s+1)K_2, K_5-e\}$-free graph with $\omega\geq 4$.
As observed in Theorem \ref{2k2k5-e}, we can see that $V(G)=A\cup C_{1,2}\cup C_{1,3}\cup C_{2,3}\cup C_{1,4}\cup C_{2,4}\cup C_{3,4}$ and for $1\leq i\leq3$, $N_A(C_{i,4})=\{v_1,v_2,v_3\}\backslash\{v_i\}$.
For $1\leq i\leq3$, we claim that $C_{i,4}$ is a union of cliques.
On contrary, for some integer $i$, $1\leq i\leq3$, if $\langle C_{i,4}\rangle$ contains an induced $P_3$, say $P$, then $\langle V(P)\cup\{v_1,v_2,v_3\}\backslash\{v_i\}\rangle\cong K_5-e$, a contradiction.
Hence for $1\leq i\leq3$, each $C_{i,4}$ is $(\omega-2)$-colorable.
Next for $1\leq i\leq\omega$, let us assign the color $i$ to the vertex $v_i$.
Since  $N_A(C_{i,4})=\{v_1,v_2,v_3\}\backslash\{v_i\}$, for $1\leq i\leq3$, the colors $\{1,4,5,\ldots,\omega\}, \{2,4,5,\ldots,\omega\}$ and $\{3,4,5,\ldots,\omega\}$ can be used for the vertices in $C_{1,4}, C_{2,4}$ and $C_{3,4}$ respectively. 
That is, altogether the $\omega$ colors $\{1,2,3,\ldots,\omega\}$ can be used for the vertices in $C_{1,4}\cup C_{2,4}\cup C_{3,4}$. 
Hence the number of new colors required for coloring the vertices in $C_{1,4}\cup C_{2,4}\cup C_{3,4}$ is at most $3(\omega-2)-\omega$. 
We know that,  $[v_2,C_{1,3}]$ is complete and $[v_1,C_{2,3}]$ is complete. 
Hence $\langle C_{1,3}\rangle$ and $\langle C_{2,3}\rangle$ are $\{sK_2, diamond\}$-free. 
In addition, $\langle C_{1,2}\rangle$ is an $\{sK_2, K_5-e\}$-free subgraph of $G$.
Hence by using our assumption and by using Corollary \ref{pk2diamond}, $\langle C_{1,2}\cup C_{1,3}\cup C_{2,3}\rangle$ is $\left(f^{\thecount}_s(\omega)+2f^{\thediamond}_s(\omega-1)\right)$-colorable. 
Therefore, $G$ is $\Big(\omega+f^{\thecount}_s(\omega)+2f^{\thediamond}_s(\omega-1)+3(\omega-2)-\omega=f^{\thecount}_s(\omega)+2f^{\thediamond}_s(\omega-1)+3\omega-6=f^{\thecount}_{s+1}(\omega)\Big)$-colorable.
\end{proof}
%
\stepcounter{count}

\subsection{Quadratic $\chi$-binding functions}\label{subsectquadratic}
In Subsection \ref{subsectquadratic}, we will consider classes of $pK_2$-free graphs for which we can establish a quadratic $\chi$-binding function. Let us start Subsection \ref{subsectquadratic} by considering the $\chi$-binding function for $\{pK_2, butterfly\}$-free graphs.
For $p\geq2$, let us define a function $f^{\thecount}_p:\mathbb{N}\rightarrow\mathbb{N}$ 
as follows. 
For $t\geq2$, $m,p\geq3$ and $s\geq 4$, $f^{\thecount}_t(1)=1$, $f^{\thecount}_2(2)=3$, $f^{\thecount}_2(3)=4$, $f^{\thecount}_2(s)=\binom{s+1}{2}$, $f^{\thecount}_p(2)=2p-2$ and $f^{\thecount}_p(m)=f^{\thecount}_{p-1}(m)+\binom{m+1}{2}-1$.
Without much difficulty, one can observe that $f_p^{\thecount}$ is a quadratic polynomial.
\begin{theorem}\label{pk2butterfly}
For $p\geq2$, if $G$ is a $\{pK_2, butterfly\}$-free graph, then $\chi(G) \leq f^{\thecount}_p(\omega(G))$.
\end{theorem}
\begin{proof}
For $\omega=1$, the result is trivial.
For $\omega=2$, the result follows from Theorem \ref{2k2} and Theorem \ref{pk2}.
Now for $\omega\geq 3$, let us prove the results by induction on $p$.
For $p=2$, by using Theorem \ref{2k2} and Theorem \ref{2k2k4}, the result holds.
For $s\geq2$,  let us assume that the result is true.

Let $G$ be an $\{(s+1)K_2, butterfly\}$-free graph.
For $1\leq i\leq \omega$, let us assign the color $i$ to the vertices in$\{v_i\}\cup I_i$.
Clearly, $\langle C_{1,2}\rangle$ is $\{sK_2, butterfly\}$-free and hence by our assumption $\langle C_{1,2}\rangle$ is $f^{\thecount}_s(\omega)$-colorable.
For $j\geq 3$, we claim that $C_{i,j}$ is an independent set.
Suppose there exist adjacent vertices $a,b\in C_{i,j}$, where $j\geq 3$, then we can find $p\in\{1,2,\ldots,j\}\backslash\{i,j\}$ such that $\langle\{a,b,v_p,v_i,v_j\}\cong butterfly$.
Hence, for $j\geq3$ and $i<j$, the vertices in $C_{i,j}$ can be colored using a single new color.
Thus 
$G$ is $\big(\omega+f^{\thecount}_s(\omega)+\sum\limits_{j=3}^{\omega}(j-1)\big)=\left(f^{\thecount}_{s}(\omega)+\binom{\omega+1}{2}-1\right)$-colorable.
\end{proof}
\stepcounter{count}
Next let us consider the $\chi$-binding function for $\{pK_2, gem^+\}$-free graphs. 
For $p\geq2$, let us define a function $f^{\thecount}_p:\mathbb{N}\rightarrow\mathbb{N}$ as follows. 
For $t\geq2$, $m,p\geq3$ and $s\geq 4$, $f^{\thecount}_t(1)=1$, $f^{\thecount}_2(2)=3$, $f^{\thecount}_2(3)=4$, $f^{\thecount}_2(s)=\binom{s+1}{2}$, $f^{\thecount}_p(2)=2p-2$ and $f^{\thecount}_p(m)=f^{\thecount}_{p-1}(m)+\binom{m+1}{2}+m-2$.
If one closely observes the recursive definition, we can see that $f_p^{\thecount}$ is quadratic.
\begin{theorem}\label{pk2gem+}
For $p\geq2$, if $G$ is a $\{pK_2, gem^+\}$-free graph, then $\chi(G) \leq f^{\thecount}_p(\omega(G))$.
\end{theorem}
\begin{proof}
Let us prove the result by induction on $p$.
For $\omega=1$, the result is trivial.
For $\omega=2$, the result follows from Theorem \ref{2k2} and Theorem \ref{pk2}.
Now for $\omega\geq 3$, let us prove the results by induction on $p$.
For $p=2$, by using Theorem \ref{2k2} and Theorem \ref{2k2k4}, the result holds.
For $s\geq2$, let us assume that the result is true.

Let $G$ be an $\{(s+1)K_2, gem^+\}$-free graph with $\omega\geq3$.
For $1\leq i\leq \omega$, let us assign the color $i$ to all the vertices in$\{v_i\}\cup I_i$.
Clearly, $\langle C_{1,2}\rangle$ is $\{sK_2, gem^+\}$-free and hence by our assumption, $\langle C_{1,2}\rangle$ is  $f^{\thecount}_s(\omega)$-colorable.
Since $\left[v_2,\cup_{j=3}^{\omega}C_{1,j}\right]$ is complete, $\omega\left(\langle\cup_{j=3}^{\omega} C_{1,j}\rangle\right)\leq\omega(G)-1$.
Also, we can observe that $\langle\cup_{j=3}^{\omega}C_{1,j}\rangle$ is $P_4$-free. 
Suppose there exists a $P_4\sqsubseteq\left\langle\cup_{j=3}^{\omega}C_{1,j}\right\rangle$, say $P$.
Then $\langle V(P)\cup\{v_2,v_1\}\rangle\cong gem^+$, a contradiction.
Thus $\langle\cup_{j=3}^{\omega}C_{1,j}\rangle$ is perfect and hence $(\omega-1)$-colorable. 
Similarly, we can prove that for $2\leq i\leq\omega-1$, $\langle\cup_{j=i+1}^{\omega}C_{i,j}\rangle$ is also perfect and $\omega\left(\langle \cup_{j=i+1}^{\omega}C_{i,j}\rangle\right)\leq\omega(G)-i+1$. Hence  for $2\leq i\leq\omega-1$, $\langle\cup_{j=i+1}^{\omega}C_{i,j}\rangle$ is $(\omega(G)-i+1)$-colorable. Altogether, 
we see that $G$ is colorable with at most $\omega+f^{\thecount}_s(\omega)+\omega-1+\sum\limits_{i=2}^{\omega-1}i=f^{\thecount}_s(\omega)+\binom{\omega+1}{2}+\omega-2=f^{\thecount}_{s+1}(\omega)$ colors.
\end{proof}
Since $dart$ is an induced subgraph of $gem^+$, for $p\geq2$, $f^{\thecount}_p(\omega(G))$ will also be the $\chi$-binding function for $\{pK_2, dart\}$-free graphs.
\begin{corollary}\label{pk2dart}
For $p\geq2$, if $G$ is a $\{pK_2, dart\}$-free graph, then $\chi(G) \leq f^{\thecount}_p(\omega(G))$.
\end{corollary}
\subsection{$\chi$-binding functions of order $O(\omega^{p-1})$}\label{subsecthigher}
\stepcounter{count}
In Subsection \ref{subsecthigher}, we will consider classes of $pK_2$-free graphs for which we have established a $\chi$-binding function of order $O(\omega^{p-1})$.
We shall begin by considering the $\chi$-binding function for $\{pK_2, K_1+C_4\}$-free graphs.
For $p\geq2$, let us define a function $f^{\thecount}_p:\mathbb{N}\rightarrow\mathbb{N}$ 
as follows. 
For $s,t\geq2$ and $m,p\geq3$, $f^{\thecount}_t(1)=1$, $f^{\thecount}_2(s)=s+1$, $f^{\thecount}_p(2)=2p-2$ and $f^{\thecount}_p(m)=\left(\sum\limits_{j=2}^{m} f^{\thecount}_{p-1}(j)\right)+f^{\thecount}_{p-1}(m-1)+1$.
\begin{theorem}\label{pk2k1+c4}
For $p\geq2$, if $G$ is a $\{pK_2, K_1 + C_4\}$-free graph, then $\chi(G) \leq f^{\thecount}_p(\omega(G))$.
\end{theorem}
\begin{proof}
Let us prove by applying Strategy 1.
Let $G$ be an $\{(s+1)K_2, K_1 + C_4\}$-free graph with $\omega\geq3$.
Let us begin by assigning the color $i$ to the vertex $v_i$ for $1\leq i\leq\omega$.
Clearly, for $(i,j)\in L$,  each $\langle C_{i,j}\rangle$ is $\{sK_2, K_1 + C_4\}$-free. 
Also, note the color $1,3,2$ can be again used while coloring the vertices in $C_{1,2}, C_{1,3}$ and $C_{2,3}$ respectively.
Thus the vertices in $C_{1,2}, C_{1,3}$ and $C_{2,3}$ can be colored with at most $f^{\thecount}_s(\omega)-1, f^{\thecount}_s(\omega-1)-1$ and $f^{\thecount}_s(\omega-1)-1$ new colors respectively.
As discussed in Theorem \ref{2k2k1+c4}, we have for $j\geq4$, $i\neq k$, $[C_{i,j},C_{k,j}]=\emptyset$ and $\cup_{i=1}^{\omega} I_i$ is an independent set.
Hence we can color the vertices in$\cup_{i=1}^{j-1} C_{i,j}$ with the color $j$ and $f^{\thecount}_s(\omega-j+2)-1$ new colors.
Finally, the vertices in$\cup_{i=1}^{\omega} I_i$ can be colored with a single new color.
Therefore, altogether $G$ can be colored with at most $\omega+f^{\thecount}_s(\omega)-1+ f^{\thecount}_s(\omega-1)-1+f^{\thecount}_s(\omega-1)-1+\left(\sum\limits_{j=4}^{\omega}f^{\thecount}_s(\omega-j+2)-1\right)+1=\left(\sum\limits_{j=2}^{\omega}f^{\thecount}_s(j)\right)+f^{\thecount}_s(\omega-1)+1=f^{\thecount}_{s+1}(\omega)$ colors.
\end{proof}
\stepcounter{count}
Next, for $p\geq2$, let us define a function $f^{\thecount}_p:\mathbb{N}\rightarrow\mathbb{N}$ (which will serve as a $\chi$-binding function for $\{pK_2, C_4\}$-free graphs) as follows.
For $s,t\geq2$ and $m,p\geq3$, $f^{\thecount}_t(1)=1$, $f^{\thecount}_2(s)=s+1$, $f^{\thecount}_p(2)=2p-2$ and $f^{\thecount}_p(m)=\left(\sum\limits_{j=2}^{m} f^{\thecount}_{p-1}(j)\right)+1$.
\begin{theorem}\label{pk2c4}
For $p\geq2$, if $G$ is a $\{pK_2, C_4\}$-free graph, then $\chi(G) \leq f^{\thecount}_p(\omega(G))$.
\end{theorem}
\begin{proof}
We shall apply Strategy 1 to prove the result. 
Let $G$ be a $\{(s+1)K_2, C_4\}$-free graph with $\omega\geq3$. Clearly $G$ is also $\{(s+1)K_2, K_1+C_4\}$-free.
Thus all the properties which was proved in Theorem \ref{pk2k1+c4} will hold for $G$ and hence $G\backslash (C_{1,2}\cup C_{1,3}\cup C_{2,3})$ can be colored with at most $\omega+\left(\sum\limits_{i=4}^{\omega}[f^{\thecount}_s(\omega-i+2)-1]\right)+1$ colors by using the coloring technique used in Theorem \ref{pk2k1+c4}.
Also $[C_{1,3}, C_{2,3}]=\emptyset$ (Otherwise there will exist vertices $a\in C_{1,3}$ and $b\in C_{2,3}$ such that $ab\in E(G)$ and hence $\langle\{a,b,v_1,v_2\}\rangle\cong C_4$, a contradiction).
Also the colors $\{1,2\}$ and the color $3$ can be used while coloring the vertices in $C_{1,2}$ and $C_{1,3}\cup C_{2,3}$ respectively.
Hence the vertices in $C_{1,2}$ and $C_{1,3}\cup C_{2,3}$ can be colored using at most $f^{\thecount}_s(\omega)-2$ new colors and $f^{\thecount}_s(\omega-1)-1$ new colors respectively.
Therefore, altogether $G$ can be colored with at most $\omega+f^{\thecount}_s(\omega)-2+\left(\sum\limits_{j=3}^{\omega}[f^{\thecount}_s(\omega-j+2)-1]\right)+1=\left(\sum\limits_{j=2}^{\omega}f^{\thecount}_s(j)\right)+1=f^{\thecount}_{s+1}(\omega)$ colors.
\end{proof}
Next, let us recall a result due to J. L. Fouquet et al. in \cite{fouquet1995graphs}.
\begin{theorem}[\cite{fouquet1995graphs}]\label{2k2house}
If $G$ is a $\{2K_2, \overline{P_5}\}$-free graph, then $\chi(G) \leq \frac{3}{2}\omega(G)$.
\end{theorem}
\stepcounter{count}
Finally, let us find a $\chi$-binding function for $\{pK_2, \overline{P_5}\}$-free graphs.
For $p\geq2$, let us define a function $f^{\thecount}_p:\mathbb{N}\rightarrow\mathbb{N}$ as follows.
For $s,t\geq2$ and $m,p\geq3$, $f^{\thecount}_t(1)=1$, $f^{\thecount}_2(s)=\frac{3}{2}s$, $f^{\thecount}_p(2)=2p-2$ and $f^{\thecount}_p(m)=m+\sum\limits_{j=2}^{m} f^{\thecount}_{p-1}(j)$.
\begin{theorem}\label{pk2house}
For $p\geq2$, if $G$ is a $\{pK_2, \overline{P_5}\}$-free graph, then $\chi(G) \leq f^{\thecount}_p(\omega(G))$.
\end{theorem}
\begin{proof}
Here also, we shall prove by using Strategy 1.
Let $G$ be an $\{(s+1)K_2, \overline{P_5}\}$-free graph with $\omega\geq3$.
For $1\leq i\leq \omega$, let us assign the color $i$ to the vertices in $\{v_i\}\cup I_i$.
Clearly, for $(i,j)\in L$,  each $\langle C_{i,j}\rangle$ is $\{sK_2, \overline{P_5}\}$-free
and hence the vertices in $C_{1,2}$ can be colored with at most $f^{\thecount}_s(\omega)$ colors.
For $j\geq 3$, we claim that  $[C_{i,j},C_{k,j}]=\emptyset$.
Suppose there exist vertices $a\in C_{i,j}$ and $b\in C_{k,j}$ such that $ab\in E(G)$, then $\langle \{a,v_k,v_i,b,v_j\}\rangle\cong \overline{P_5}$, a contradiction.
Thus for each $j\geq3$, the vertices in $\cup_{i=1}^{j-1}C_{i,j}$ can be colored using $f^{\thecount}_s(\omega-j+2)$ colors.
Therefore, altogether $G$ can be colored with at most $\omega+f^{\thecount}_s(\omega)+\sum\limits_{j=3}^{\omega}f^{\thecount}_s(\omega-j+2)
=\omega+\sum\limits_{j=2}^{\omega}f^{\thecount}_s(j)=f^{\thecount}_{s+1}(\omega)$ colors.
\end{proof}
\subsection*{Acknowledgment}
\small For the first author,  this research was supported by the Council of Scientific and Industrial Research,  Government of India, File No: 09/559(0133)/2019-EMR-I.
And for the second author, this research was supported by SERB DST, Government of India, File no: EMR/2016/007339. Also, for the third author, this research was supported by the UGC-Basic Scientific Research, Government of India, Student id: gokulnath.res@pondiuni.edu.in. 
\bibliographystyle{ams}
\bibliography{ref}
\end{titlepage}
\listoftodos
\end{document}